\documentclass[12pt]{amsart}
\usepackage{amssymb,amsmath,amsthm,latexsym}
\usepackage[rgb,usenames,dvipsnames,pdftex]{xcolor}
\usepackage{tikz}
\usepackage{mathrsfs}
\usepackage{mdwlist,enumitem}
\usepackage{a4wide}
\usepackage{graphicx}
\usepackage{mathtools,dsfont}
\usepackage{mathdots}
\usepackage{relsize}

\definecolor{c4}{hsb}{0,1,0}

\newcommand{\R}{\mathbb{R}}

\newcommand{\N}{\mathbb{N}}

\newcommand{\La}{\mathbb{L}}

\newcommand{\NN}{\mathcal{N}}

\newcommand{\e}{\varepsilon}

\newcommand{\abs}[1]{\lvert#1\rvert}

\renewcommand{\leq}{\leqslant}
\renewcommand{\geq}{\geqslant}

\newtheorem*{corollary*}{Corollary}

\newtheorem{theorem}{Theorem}
\newtheorem{lemma}[theorem]{Lemma}

\newtheorem{cor}[theorem]{Corollary}

\theoremstyle{definition}

\newtheorem{example}[theorem]{Example}

\theoremstyle{remark}

\begin{document}
\title[Approximate homomorphisms on lattices]{Approximate homomorphisms on lattices}

\author[R. Badora]{Roman Badora}
\address{Institute of Mathematics, University of Silesia, Bankowa 14, 40-007 Katowice, Poland}
\email{robadora@ux2.math.us.edu.pl}

\author[T. Kochanek]{Tomasz Kochanek}
\address{Institute of Mathematics, Polish Academy of Sciences, \'Sniadeckich 8, 00-656 Warsaw, Poland\, {\rm and}\, Institute of Mathematics, University of Warsaw, Banacha~2, 02-097 Warsaw, Poland}
\email{tkoch@impan.pl}

\author[B. Przebieracz]{Barbara Przebieracz}
\address{Institute of Mathematics, University of Silesia, Bankowa 14, 40-007 Katowice, Poland}
\email{barara.przebieracz@us.edu.pl}


\begin{abstract}
We prove two results concerning an~Ulam-type stability problem for homomorphisms between lattices. One of them involves estimates by quite general error functions; the~other deals with approximate (join) homomorphisms in terms of certain systems of lattice neighborhoods. As a~corollary, we obtain a~stability result for approximately monotone functions.
\end{abstract}
\maketitle

\section{Introduction}
More than half a century ago, S.M.~Ulam \cite{U} posed the problem of finding conditions which guarantee that any nearly additive map defined, say, on a~semigroup must be close to a~truly additive map. In 1978, P.M.~Gruber \cite{G} reformulated his question by posing a~more general stability problem: ``{\it Suppose a~mathematical object satisfies a~certain property approximately. Is it then possible to approximate this object by objects, satisfying the property exactly?}''. This initiated a~broad research program on the~stability problem in theory of functional equations; for more information the~reader may consult \cite{HR}.

In this note, we deal with an~Ulam-type problem for homomorphisms of lattices. We present two results where satisfying the~homomorphism equation `approximately' is formalized either with the~aid of error functions or in terms~of abstractly understood neighborhoods in lattices. In order to justify our approach let us mention a~few known results concerning the~stability problem in lattices. 

A~pioneering paper in this context is the~one by N.J.~Kalton and J.W.~Roberts \cite{KR} which contains the following deep result (originally formulated for~algebras of sets and nearly additive set functions).
\begin{theorem}[Kalton \& Roberts \cite{KR}]\label{KR}
Let $X$  be a Boolean algebra and $f\colon X\to\R$ a~function satisfying
$$
|f(x\vee y)-f(x)-f(y)|\leq1\quad\mbox{for }x,y\in X\mbox{ with }x\wedge y=0.
$$ 
Then there exists a~map $g\colon X\to \R$ such that 
$$
g(x\vee y)=g(x)+g(y)\quad\mbox{ for }x,y\in X\mbox{ with }x\wedge y=0
$$ 
and $\abs{f(x)-g(x)}< 45$ for every $x\in X$.
\end{theorem}\noindent
This result is of fundamental importance in functional analysis, especially in theory of twisted sums of quasi-Banach spaces (see, {\it e.g.} \cite{kalton}), as well as in the stability problem for vector meaures ({\it cf. }\cite{koch}). A~somehow related result, very combinatorial in its nature, was obtained by I.~Farah.
\begin{theorem}[Farah \cite{F}] 
Let $n,m\in\N$, $X=2^{\{1,2,\ldots,m\}}$ and $Y=2^{\{1,2,\ldots, n\}}$. 
Suppose that $\varphi\colon Y\to [0,\infty]$ is a submeasure, that is, $\varphi (0)=0$, $\varphi(A)\leq \varphi(A\cup B)$, for $A,B\in Y$ and $\varphi (A\cup B)\leq \varphi(A)+\varphi(B)$, for $A,B\in Y$. Moreover, we assume that $\phi$ is nonpathological, that is, it is equal to the~supremum of all measures it dominates. Let $\e>0$ and $f\colon X\to Y$ be such that
$$
\varphi(f(x\cup y)\div (f(x)\cup f(y))<\e\quad \mbox{for }x,y\in X,
$$
$$
\varphi(f(X\setminus x)\div (Y\setminus f(x)))<\e\quad\mbox{for }x\in X.
$$ 
Then there exists a~lattice homomorphism $g\colon X\to Y$ such that $\varphi(f(x)\div g(x))<521\e$ for every $x\in X$.
\end{theorem} 

In the light of the result above, one seemingly natural approach to the stability problem in lattices would be to assume that a~given map $f\colon X\to Y$ between lattices $X$ and $Y$ satisfies $f(x\vee y)\div (f(x)\vee f(y))\leq\e$ for all $x,y\in X$ and some fixed $\e\in Y$. However, such an~approach turns out to be too naive, at least in the~case where $Y$ is assumed to be a~Boolean algebra. Indeed, just define $g(x)=f(x)\setminus \e$ and notice that $g$ is then a~lattice homorphism satisfying $f(x)\div g(x)\leq\e$ for each $x\in X$. Therefore, we propose two different ways of formalizing the stability problem---one involves error functions instead of the~constant factor $\e$, while the~other expresses the~closedness of $f(x\vee y)$ to $f(x)\vee f(y)$ in terms of a~system of neighborhoods (see Theorems~\ref{main} and \ref{main2} below, respectively). 

Both of our approaches share a~common root which is a~certain separation, or a~`sandwich-type', result (see Lemma~\ref{sep-thm} below). It is related to some already known separation theorems.
\begin{theorem}[F\"{o}rg-Rob, Nikodem, P\'{a}les \cite{FRNP}]\label{FRNP1} 
Assume that a~function  $f\colon\mathbb{R}\to\mathbb{R}$ is quasi-concave {\rm (}i.e. $f(x)\geq\min \{ f(a), f(b)\}$ for all $a\leq x\leq b${\rm )}, a~function $g
\colon \mathbb{R}\to\mathbb{R}$ is quasi-convex {\rm (}i.e. $g(x)\leq\max \{ g(a), g(b)\}$ for all $a\leq x\leq b${\rm )} and we have $f(x)\leq g(x)$ for every $x\in\R$. Then there exists a~monotone map $h\colon\mathbb{R}\to\mathbb{R}$ such that $f(x)\leq h(x)\leq g(x)$ for every $x\in\R$.
\end{theorem} 
W. Kubi\'{s} \cite{K} noted that a similar `sandwich-type' theorem is valid for maps between linearly ordered spaces (see \cite[Thm.~2.1]{K}), whereas it fails to hold for maps from $\mathbb{R}^2$ to $\mathbb{R}$ and actually even for maps from the~four-element Boolean algebra $\{0,1\}^2$ to the~three-element linearly ordered space $\{0,1,2\}$. Moreover, Kubi\'{s} showed (\cite[Thm.~3.3]{K}) that a~`sandwich-type' theorem for the class of $S_{4}$ bi-convexity spaces holds true when the image space is a~complete Boolean algebra (being a~bi-convexity space with convexities consisting of ideals and filters). The~class of $S_4$ bi-convexity spaces includes, for example, real vector spaces (for more information consult \cite{V}). As a~~
consequence of this quite abstract version of a~separation theorem, Kubi\'{s} derived the~following corollary (see \cite[Thm.~3.7]{K}).
\begin{theorem}[Kubi\'{s} \cite{K}]\label{Kt} 
Let $L$ be a distributive lattice, $\mathbb{B}$ be a complete Boolean algebra, $f, g\colon L \to\mathbb{B}$ and assume that $f$ is a~meet homomorphism, $g$ is a~join homomorphism and $f(x)\leq g(x)$ for $x\in L$. Then there exists a~lattice homomorphism $h\colon L \to\mathbb{B}$ such that $f(x)\leq h(x)\leq g(x)$ for every $x\in L$. 
\end{theorem}

As an~application of their Theorem~\ref{FRNP1}, F\"{o}rg-Rob, Nikodem and P\'{a}les showed the following result which yields an~Ulam-type stability for monotone maps.
\begin{theorem}[F\"{o}rg-Rob, Nikodem, P\'{a}les \cite{FRNP}]\label{Pal} 
Let $I\subset \R$ be an interval, $\e\geq 0$ and assume that a~function $f\colon I\to \R$ satisfies
$$
\min\{f(x),f(y)\}-\e\leq f(tx+(1-t)y)\leq \max\{f(x),f(y)\}+\e\quad\mbox{for } x,y\in I,\, t\in[0,1].
$$ 
Then there exists a monotone function $g\colon I\to \R$ such that $|f(x)-g(x)|\leq\e/2$ for every $x\in I$.
\end{theorem}\noindent
We will see that this theorem can be easily derived from our results on approximate lattice homomorphisms.

\section{Results}
Recall that a lattice is called {\it conditionally complete} provided every its bounded subsets admits the~least upper bound and the~greatest lower bound. A~map $f$ between lattices $X$ and $Y$ is called a~{\it join homomorphism} if it preserves joins, {\it i.e.} $f(x\vee y)=f(x)\vee f(y)$ for $x,y\in X$, and it is called a~{\it meet homomorphism} if it preserves meets, {\it i.e.} $f(x\wedge y)=f(x)\wedge f(y)$ for $x,y\in X$. It is called a~{\it lattice homomorphism} if it is both a~join and a~meet homomorphism. In the proofs of our stability theorems we shall need the~following separation result.
\begin{lemma}\label{sep-thm} 
Let $X$ be a distributive lattice and $Y$ be a conditionally complete lattice. Assume that maps $\Phi,\Psi\colon X\to Y$ satisfy the following conditions: $\Phi\leq \Psi$,
$$\Phi(x\vee y)\leq\Phi(x)\vee\Phi(y)\quad\mbox{for } x,y\in X,$$
$$\Psi(x\vee y)\geq \Psi(x)\vee \Psi(y)\quad\mbox{for } x,y\in X.$$ 
Then there exists a join homomorphism $F\colon X\to Y$ such that $\Phi\leq F\leq \Psi$.
\end{lemma}

\begin{proof}
Notice that for each $x\in X$ and each $z\in X$ with $z\leq x$ we have 
\begin{equation}\label{1}
\Phi(z)\leq\Psi(z)\leq\Psi(z)\vee\Psi(x)\leq\Psi(z\vee x)=\Psi(x).
\end{equation} 
Hence, the set $\{\Phi(z):\ z\leq x\}$ is bounded from above and we can define 
\begin{equation}\label{defF}
F(x)=\sup\{\Phi(z)\colon z\leq x\}\quad\mbox{for } x\in X.
\end{equation} 
Plainly, we have $F(x)\geq \Phi(x)$, while \eqref{1} implies that also $F(x)\leq\Psi(x)$ for every $x\in X$.
 
Now, fix any $x,y\in X$ and consider an~arbitrary $z\leq x\vee y$. Since 
$$
z=z\wedge (x\vee y)=(z\wedge x)\vee (z\wedge y),
$$
we have
$$
\Phi(z)\leq \Phi(z\wedge x)\vee\Phi(z\wedge y)\leq F(x)\vee F(y).
$$ 
From the definition of $F$ we thus get that $F(x\vee y)\leq F(x)\vee F(y)$. Moreover, $F$ is a~monotone increasing function, therefore,  
$F(x)\leq F(x\vee y)$ and $F(y)\leq F(x\vee y)$ which implies that $F(x)\vee F(y)\leq F(x\vee y)$ and finishes the~proof.
\end{proof}

Notice that by interchanging $\vee$ and $\wedge$ in the~lattice $X$ or $Y$ (or both), we can derive three analogous results to Lemma~\ref{sep-thm}. Moreover, combining this lemma with the Kubi\'s result (Theorem~\ref{Kt}), we infer that in the case where $L$ is a~distributive lattice, $\mathbb{B}$ is a~complete Boolean algebra and maps $\Phi_{1}, \Phi_{2},\Psi_{1},\Psi_{2}\colon L\to \mathbb{B}$ satisfy $\Psi_{2}\leq\Phi_{2}\leq\Phi_{1}\leq \Psi_{1}$ along with the~following system of inequalites: 
$$\Psi_{2}(x\wedge y)\leq \Psi_{2}(x)\wedge  \Psi_{2}(y)$$
$$\Phi_{2}(x\wedge y)\geq\Phi_{2}(x)\wedge \Phi_{2}(y)$$
$$\Phi_{1}(x\vee y)\leq\Phi_{1}(x)\vee\Phi_{1}(y)$$
$$\Psi_{1}(x\vee y)\geq \Psi_{1}(x)\vee \Psi_{1}(y)$$
for all $x,y\in L$, then there is a~lattice homomorphism $H\colon L \to\mathbb{B}$ lying between $\Psi_{2}$ and $\Psi_{1}$, {\it i.e.} $\Psi_2\leq H\leq\Psi_1$.

We are in a position to prove our first stability result.
\begin{theorem}\label{main}
Let $X$ and $Y$ be distributive lattices and assume that $Y$ is conditionally complete and satisfies the~dual to the~infinite distributive law, that is, 
$$
y\vee\inf S=\inf\{y\vee s\colon s\in S\}
$$
for every $y\in Y$ and $S\subset Y$ bounded from below. Assume that maps $f\colon X\to Y$, $\phi,\psi\colon X\times X\to Y$ satisfy the~following conditions: 
\begin{equation}\label{11}
\phi(z,z)\leq\phi(x,y)\quad\mbox{for }x,y,z\in X\mbox{ with }x,y\leq z,
\end{equation}
\begin{equation}\label{4}
\psi(x,y)\leq\psi(z,z)\quad\mbox{for }x,y,z\in X\mbox{ with }x,y\leq z
\end{equation}
and
\begin{equation}\label{2}
\phi(x,y)\wedge f(x\vee y)\leq f(x)\vee f(y)\leq f(x\vee y)\vee \psi(x,y)\quad\mbox{for }x,y\in X.
\end{equation} 
Then there exists a join homomorphism $F\colon X\to Y$ such that
\begin{equation}\label{7}
\phi(x,x)\wedge f(x)\leq F(x)\leq f(x)\vee \psi(x,x)\quad\mbox{for }x\in X.
\end{equation} 
\end{theorem}

\begin{proof}
We start by proving that
\begin{equation}\label{3}
\phi(x,x)\wedge f(x)\leq f(x_1)\vee\ldots\vee f(x_n)\leq f(x)\vee\psi(x,x)
\end{equation}
for all $x_1,\ldots,x_n\in X$, where $x=x_1\vee\ldots\vee x_n$. 

We proceed by induction on $n$. For $n=1$ the~inequality is obvious. Assume \eqref{3} holds for some $n\in \N$ and all $x_1,\ldots,x_n\in X$. For arbitrary $x_1,\ldots,x_{n+1}\in X$ set $x=x_1\vee\ldots\vee x_n$ and $\bar{x}=x\vee x_{n+1}$. By the~induction hypothesis, \eqref{2} and \eqref{4}, we obtain
\begin{equation*}
\begin{split}
(f(x_1)\vee\ldots\vee f(x_n))\vee f(x_{n+1}) &\leq f(x)\vee \psi(x,x)\vee f(x_{n+1})\\
&\leq f(\bar{x})\vee \psi(x,x_{n+1})\vee \psi(x,x)\\
&\leq f(\bar{x})\vee \psi(\bar{x},\bar{x}).
\end{split}
\end{equation*}
Similarly, using \eqref{2} and \eqref{11} we get
\begin{equation*}
\begin{split}
(f(x_1)\vee\ldots\vee f(x_n))\vee f(x_{n+1}) &\geq (f(x)\wedge \phi(x,x))\vee f(x_{n+1})\\
&=(\phi(x,x)\vee f(x_{n+1}))\wedge (f(x)\vee f(x_{n+1}))\\
&\geq \phi(x,x)\wedge \phi(x,x_{n+1})\wedge f(\bar{x})\\
&\geq \phi(\bar{x},\bar{x})\wedge f(\bar{x})
\end{split}
\end{equation*}
which completes the~inductive proof of inequality~\eqref{3}.

Define functions $\Phi,\Psi\colon X\to Y$ be the~formulas 
$$\Phi(x)=\inf\bigl\{f(x_1)\vee\ldots\vee f(x_n)\colon n\in\N,\,  x_1,\ldots,x_n\in X,\, x=x_1\vee\ldots\vee x_n\bigr\}
$$
and
$$
\Psi(x)=\sup\bigl\{f(x_1)\vee\ldots\vee f(x_n)\colon n\in\N,\,  x_1,\ldots,x_n\in X,\, x=x_1\vee\ldots\vee x_n\bigr\}.
$$
Note that these definitions are correct as inequality~\eqref{3} guarantees that the~set under the~infimum and the~supremum sign is bounded. Note also that the~same inequality implies that
\begin{equation}\label{9}
\phi(x,x)\wedge f(x)\leq \Phi(x)\leq f(x)\quad\mbox{for }x\in X;
\end{equation} 
\begin{equation}\label{10}
f(x)\leq \Psi(x)\leq f(x)\vee \psi(x,x)\quad\mbox{for }x\in X.
\end{equation}
Moreover, notice that 
\begin{equation}\label{sub}
\Phi(x\vee y)\leq \Phi(x)\vee \Phi(y)\quad\mbox{for }x,y\in X;
\end{equation}
\begin{equation}\label{5}
\Psi(x\vee y)\geq \Psi(x)\vee \Psi(y)\quad\mbox{for }x,y\in X.
\end{equation} 
Indeed, inequality \eqref{sub} follows from the~assumed dual distributivity law and the~fact that for arbitrary $x_1,\ldots,x_n,y_1\ldots y_m\in X$ satisfying $x=x_1\vee\ldots\vee x_n$ and $y=y_1\vee\ldots\vee y_m$ we have
$$
\Phi(x\vee y)\leq (f(x_1)\vee \ldots\vee f(x_n))\vee(f(y_1)\vee\ldots\vee f(y_m)).
$$ 
For inequality \eqref{5} observe that for any $x_1,\ldots,x_n$ as above we have
$$
f(x_1)\vee\ldots\vee f(x_n)\leq f(x_1)\vee\ldots\vee f(x_n)\vee f(y)\leq \Psi(x\vee y)
$$
which yields $\Psi(x)\leq\Psi(x\vee y)$. Similarly, we get $\Psi(y)\leq \Psi(x\vee y)$, hence inequality \eqref{5}.
 
Finally, an appeal to Lemma~\ref{sep-thm} produces a~join homomorphism $F\colon X\to Y$ such that $\Phi\leq F\leq\Psi$. Combining it with \eqref{9} and \eqref{10} we obtain assertion \eqref{7} as desired.
\end{proof}
 
In our next result we express the assumption that $f(x)\vee f(y)$ is `close' to $f(x\vee y)$ with the~aid of a~system of lattice neighborhoods.
\begin{theorem}\label{main2}
Let $X$ and $Y$ be distributive lattices and assume that $Y$ is conditionally complete and satisfies the~dual to the~infinite distributive law. Assume moreover that there is a~function $\NN\colon  Y\to2^Y$, each of whose value is a~bounded set, and which satisfies the following conditions:
\begin{itemize}
\item[{\rm (i)}] $y\in\NN(y)$ for each $y\in Y$;
\item[{\rm (ii)}] if $t,u\in\NN(z)$ and $t\leq y\leq u$, then $y\in\NN(z)$;
\item[{\rm (iii)}] $\sup\NN(y)\in \NN(y)$ and $\inf \NN(y)\in \NN(y)$ for each $y\in Y$;
\item[{\rm (iv)}] if $t\in \NN(u)$ and $u\vee y\in \NN(z)$, then $t\vee y\in \NN(z)$.
\end{itemize} 
Then for every map $f\colon X\to Y$ satisfying 
\begin{equation}\label{bcnd}
f(x)\vee f(y)\in\NN(f(x\vee y))\quad\mbox{for }x,y\in X
\end{equation}
there exists a~join homomorphism $F\colon X\to Y$ such that $F(x)\in \NN(f(x))$ for every $x\in X$. 
\end{theorem}

\begin{proof} By induction we can show that 
\begin{equation}\label{indN}
f(x_1)\vee\ldots\vee f(x_n)\in\NN(f(x_1\vee\ldots\vee x_n))
\end{equation}
for all $n\in\N$ and $x_1,\ldots,x_n\in X$. Indeed, condition (i) gives the assertion for $n=1$. Suppose condition \eqref{indN} is valid for a~fixed $n\in\N$ and let $x_1,\ldots,x_{n+1}\in X$. Set
$$
a=f(x_1)\vee\ldots\vee f(x_n),\,\,\,\, b=f(x_1\vee\ldots\vee x_n),\,\,\,\, c=f(x_{n+1}),\,\,\,\, d=f(x_1\vee\ldots\vee x_{n+1}).
$$
By condition \eqref{bcnd}, we have $b\vee c\in\NN(d)$, whereas our inductive hypothesis gives $a\in\NN(b)$. Therefore, by condition (iv) we infer that $a\vee c\in\NN(d)$ which completes the~inductive proof of \eqref{indN}.

Now, we define maps $\Phi,\Psi\colon X\to Y$ as in the~previous proof, that is,
$$\Phi(x)=\inf\bigl\{f(x_1)\vee\ldots\vee f(x_n)\colon n\in\N,\,  x_1,\ldots,x_n\in X,\, x=x_1\vee\ldots\vee x_n\bigr\},
$$
$$
\Psi(x)=\sup\bigl\{f(x_1)\vee\ldots\vee f(x_n)\colon n\in\N,\,  x_1,\ldots,x_n\in X,\, x=x_1\vee\ldots\vee x_n\bigr\}.
$$
Note that, by \eqref{indN}, every element of the set under the~infimum/supremum sign belongs to $\NN(f(x))$. Hence, assumptions (ii) and (iii) imply that
\begin{equation}\label{psiphi}
\Psi(x),\Phi(x)\in\NN(f(x))\quad\mbox{for }x\in X.
\end{equation}
Moreover, we have $\Phi\leq f\leq\Psi$ and similarly as in the~proof of Theorem~\ref{main} we show that $\Phi(x\vee y)\leq \Phi(x)\vee \Phi(y)$ and $\Psi(x\vee y)\geq \Psi(x)\vee \Psi(y)$ for all $x,y\in X$.

In view of Lemma \ref{sep-thm}, there is a~join homomorphism $F\colon X\to Y$ such that $\Phi\leq F\leq\Psi$. Finally, by \eqref{psiphi} and assumption (ii), we obtain $F(x)\in\NN(f(x))$ for every $x\in X$.
\end{proof}

Notice that by interchanging $\vee$ and $\wedge$ in $X$ or $Y$ (or both) we can obtain three analogous results to Theorems~\ref{main} and \ref{main2}.

\begin{example}
Condition (ii) above claims, in the lattice terminology, that each $\NN(z)$ is {\it convex}. By (iii), we require that it is closed under the~sup/inf operations. Condition (iv) seems a~bit demanding, however, it is satisfied in some natural situations as the~examples below show.

\vspace*{2mm}\noindent
{\bf (a) }Assume $Y$ has the~minimal element. For each $z\in Y$ define $\NN(z)=\{y\in Y\colon y\leq z\}$ which is nothing else but the~ideal generated by $\{z\}$. It is easily verified that the~function $Y\ni z\mapsto\NN(z)$ satisfies all the~axioms (i)--(iv).

\vspace*{2mm}\noindent
{\bf (b) }For any $n\in\N$, consider the lattice $\La_n$ consisting of all natural divisors of $n$, that is, $(\La_n,\vee,\wedge)$ is given as $\La_n=\{k\in\N\colon k\mid n\}$, where $\vee$ is the least common multiple and $\wedge$ is the greatest common divisor. For $z\in\La_n$ define $\NN(z)$ as the~family of those elements of $\La_n$ whose each prime factor is also a~prime factor of $z$. Again, the~function $\La_n\ni z\mapsto\NN(z)$ satifies (i)--(iv).

\vspace*{2mm}\noindent
{\bf (c) }We can repeat the same idea as above for every finite distributive lattice $Y$. Let $H\colon Y\to\La_n$ be a~one-to-one homomorphism, for a~suitable $n\in\N$ (see, {\it e.g.}, \cite[Ch.~II.1]{gratzer}), and for every $z\in Y$ define $\NN(z)$ as the~family of those $y\in Y$ for which we have $H(y)\in\NN(H(z))$ in the~sense of the~definition from the~previous example.

\vspace*{2mm}\noindent
{\bf (d) }Let $\Theta\subset Y^2$ be a~congruence relation (see \cite[Ch.~I.3]{gratzer}) and assume it is {\it conditionally complete} in the~sense that if $y_i\equiv z_i (\Theta)$ for $i\in I$, then both $\{y_i\colon i\in I\}$ and $\{z_i\colon i\in I\}$ are bounded, and we have $\bigvee_{i\in I}y_i\equiv\bigvee_{i\in I}z_i(\Theta)$ and $\bigwedge_{i\in I}y_i\equiv\bigwedge_{i\in I}z_i(\Theta)$. For $z\in Y$ define $\NN(z)$ to be the~abstraction class determined by $z$, $\NN(z)=\{y\in Y\colon y\equiv z(\Theta)\}$. Conditions (i)--(iv) are then satisfied. Indeed, (i) is trivial; (ii) is valid since every abstraction class forms a~convex sublattice (see \cite[Lemma~I.3.7]{gratzer}); (iii) follows from the~conditional completeness of $\Theta$; (iv) follows from the~fact that $\Theta$, as every congruence, preserves the~join operation.
\end{example}

\section{Approximate monotonicity} 
In this section we will show how our theorem can be applied to obtain 
a~stability result for approximately monotone functions. First, note that obviously for any $D\subset\R$ a~function $f\colon D\to \R$ is increasing if and only if $\max\{f(x),f(y)\}=f(\max\{x,y\})$ for all $x,y\in D$.
\begin{cor}\label{cor21} 
Let $D\subset\R$, $\e\geq 0$ and assume that a~function $f\colon D\to\R$ satisfies
\begin{equation}\label{apprinc}
\max\{f(x),f(y)\}-f(\max\{x,y\})\leq\e\quad\mbox{for }x,y\in D.
\end{equation}
Then there exists an~increasing function $g\colon D\to \R$ such that $\abs{f(x)-g(x)}\leq\e/2$ for every $x\in D$.
\end{cor}

\begin{proof} 
We consider the lattices $X=D$ and $Y=[-\infty,\infty]$ in which the~join and the~meet operations are defined by $x\vee y=\max\{x,y\}$ and $x\wedge y=\min\{x,y\}$. Obviously, $X$ and $Y$ are distributive, $Y$ is also conditionally complete and satisfies the~dual to the~infinite distributive law.

Define $\phi, \psi\colon X\times X\to Y$ by 
$$\phi(x,y)=-\infty\quad\mbox{and}\quad \psi(x,y)=\sup\{f(z)\colon z\leq x\}\vee \sup\{f(z)\colon z\leq y\}.
$$
Inequality \eqref{11} is then obvious. Note that for all $x,y,z\in X$ with $x\leq z$ and $y\leq z$ we have
$$
\psi(x,y)\leq\sup\{f(u)\colon u\leq z\} =\psi(z,z)
$$
which means that inequality \eqref{4} is also satisfied.

Now, we shall verify that inequalities \eqref{2} are satisfied. The left one is obvious. For the right one, fix any $x,y\in X$ and assume with no loss of generality that $y\leq x$. Then 
$$
f(x)\leq f(x)\vee\psi(x,y)=f(x\vee y)\vee\psi(x,y)
$$ 
and
$$
f(y)\leq \sup\{f(z)\colon z\leq x\}\leq \psi(x,y)\leq f(x\vee y)\vee \psi(x,y),
$$ 
as desired.

By Theorem \ref{main}, there is a~join homomorphism ({\it i.e.} an~increasing function) $F\colon X\to Y$ which satisfies condition \eqref{7}. Recall that the~map $F$ was defined by the~formula $F(x)=\sup\{\Phi(z)\colon z\leq x\}$, where
$$
\Phi(x)=\inf\bigl\{f(x_1)\vee\ldots\vee f(x_n)\colon n\in\N,\,  x_1,\ldots,x_n\in X,\, x=x_1\vee\ldots\vee x_n\bigr\}
$$
({\it cf.} the proofs of Lemma~\ref{sep-thm} and Theorem~\ref{main}). From \eqref{7} and \eqref{apprinc} we get
$$
F(x)\leq f(x)\vee \psi(x,x)=f(x)\vee \sup\{f(z)\colon z\leq x\} \leq f(x)+\e\quad\mbox{for }x\in X.
$$
Notice that if $x=x_1\vee \ldots\vee x_n$ for some $x_1,\ldots,x_n\in X$, then $x=x_i$ for some $i\in\{1,\ldots,n\}$. Hence, by the~very definition, we have $\Phi(x)\geq f(x)$ and therefore $F(x)\geq\Phi(x)\geq f(x)$ for each $x\in X$. We have shown that $F$ is an~increasing, real-valued function such that $f(x)\leq F(x)\leq f(x)+\e$ for each $x\in X$. It remains to define $g\colon X\to Y$ as $g(x)=F(x)-\e/2$.
\end{proof}

Observe that interchanging $\vee$ and $\wedge$ in the above proof we obtain an~analogous result on approximately decreasing functions.
\begin{cor} \label{cor22}
Let $D\subset\R$, $\e\geq 0$ and assume that a~function $f\colon D\to\R$ satisfies
\begin{equation}\label{apprinc}
f(\max\{x,y\})-\min\{f(x),f(y)\}\leq\e\quad\mbox{for }x,y\in D.
\end{equation}
Then there exists a~decreasing function $g\colon D\to \R$ such that $\abs{f(x)-g(x)}\leq\e/2$ for every $x\in D$.
\end{cor}

Finally, notice that Theorem~\ref{Pal} can be easily derived from Corollaries~\ref{cor21} and \ref{cor22}, as well as these two corollaries can be derived from Theorem~\ref{Pal} (in the~case where $D$ is an~interval).

\proof[Acknowledgement] The research of the third-named author is a part of the~{\it Iterative functional equations and real analysis} program (Institute of Ma\-the\-ma\-tics, University of Silesia, Katowice, Poland).

\end{document}